\documentclass[a4paper,11pt]{amsart}

\usepackage{amsmath,url}
\usepackage{amssymb}
\usepackage{color}
\usepackage{graphicx}
\usepackage{bbm}
\usepackage{a4wide}
\usepackage{t1enc}
\usepackage{amsthm}
\newtheorem{theorem}{Theorem}[section]
\newtheorem{lemma}[theorem]{Lemma}
 \newtheorem{prop}[theorem]{Proposition}
\newtheorem{cor}[theorem]{Corollary}

\theoremstyle{definition}
\newtheorem{definition}[theorem]{Definition}
\newtheorem{example}[theorem]{Example}
\newtheorem{remark}[theorem]{Remark}
\newtheorem{fact}[theorem]{Fact}

\numberwithin{theorem}{section}

\newcommand{\sos}{\operatorname{SOS}}
\newcommand{\soc}{\operatorname{SOC}}
\newcommand{\oc}{\operatorname{OC}}
\newcommand{\os}{\operatorname{OS}}

\newcommand{\ocg}{\operatorname{OC(\varGamma)}}
\newcommand{\osg}{\operatorname{OS(\varGamma)}}
\newcommand{\so}{\operatorname{SO}}
\newcommand{\og}{\operatorname{O}}
\newcommand{\gl}{\operatorname{GL}}
\newcommand{\g}{\varGamma}
\newcommand{\OO}{\mathcal O}
\newcommand{\sca}{\operatorname{scal}_\varGamma}

\newcommand{\rk}{\operatorname{rank}}

\newcommand{\re}{\right}
\newcommand{\li}{\left}

\newcommand{\NN}{\mathbbm{N}}
\newcommand{\CC}{\mathbbm{C}}
\newcommand{\QQ}{\mathbbm{Q}}
\newcommand{\LQ}{\mathbbm{Q}}
\newcommand{\RR}{\mathbbm{R}}
\newcommand{\ZZ}{\mathbbm{Z}}
\newcommand{\LZ}{\mathbbm{Z}}
\newcommand{\s}{\mathbb S}

\newcommand{\klg}{\leqslant}
\newcommand{\grg}{\geqslant}

\DeclareMathOperator{\Ker}{Ker}

\newcommand{\bmidb}[2]{\{\nonscript\,{#1}\mid{#2}\nonscript\,\}}

\newdimen{\standardlabelwidth}
\standardlabelwidth-\standardlabelwidth
  \advance\standardlabelwidth by 2\normalparindent
\newcommand{\standardlabel}[1]{#1\kern\standardlabelwidth}

\newenvironment{enumerateA}{%
\setlength{\labelsep}{0pt}%
\setlength{\leftmargini}{0pt}%
\begin{enumerate}%
\setlength{\labelwidth}{0pt}%
\setlength{\itemindent}{0pt}%
}{\end{enumerate}}

\begin{document}
\title{Similarity and Coincidence Isometries for Modules}
\author{Svenja Glied}
\address{Fakult\"at f\"ur Mathematik, Universit\"at Bielefeld, Postfach 100131, 33501 Bielfeld, 
       Germany}
\curraddr{}
\email{sglied@math.uni-bielefeld.de}

\thanks{}

\begin{abstract}
          The groups of (linear) similarity and coincidence isometries of certain modules $\g$ in 
          $d$-dimensional Euclidean space, which naturally occur in quasicrystallography, are considered.
          It is shown that the structure of the factor group of similarity modulo coincidence isometries 
          is the direct sum of cyclic groups of prime power orders that divide $d$.
          In particular, if the dimension $d$ is a prime number $p$, the factor group is an elementary 
          Abelian $p$-group. This generalizes previous results obtained for lattices to situations 
          relevant in quasicrystallography.

\end{abstract}

\maketitle
MSC: 20H15, 82D25, 52C23 

\section{Introduction}\label{intro}
The classification of colour symmetries and that of grain boundaries in crystals 
and quasicrystals are
closely related to the existence of similar and coincidence
sublattices of the underlying lattice of periods or the corresponding
translation module; cf.~\cite{baake,Gri}. 
 It is thus of interest to understand the
corresponding groups of isometries from a more mathematical perspective.
For a free $\ZZ$-module $M\subset \RR^d$ of finite rank that spans $\RR^d$, an element 
$R\in \og(d,\RR)$ is called a 
{\em coincidence isometry of $M$} if $RM$ and $ M$ are commensurate, written $RM \sim M$, which means 
that their intersection has finite index both in $M$ and in $RM$. We let $\oc(M)$ denote the set of all 
coincidence isometries of $M$. 
More generally, a {\em similarity isometry of $M$} is an element
$T \in \og(d,\RR)$ with $\alpha T M \sim M$ for some positive real number $\alpha$. This definition was 
first introduced for lattices in \cite{Heu}. The set $\os(M)$ of all 
similarity isometries of $M$ obviously contains $\oc(M)$ as a subset.

For subrings $\mathcal{S}$ of the rings of integers of real algebraic number fields, we consider 
the similarity and coincidence isometries of free $\mathcal S$-modules $\g\subset \RR^d$ of rank $d$ 
that span $\RR^d$. 
For a separate treatment of the crystallographic case $\mathcal S = \ZZ$, where $\g$ is a lattice, the 
reader is referred to \cite{Gli}.
We show that the similarity isometries of $\g$ form a group that  
contains the coincidence isometries as a normal subgroup. 
The corresponding factor group of similarity modulo coincidence isometries is the direct 
sum of cyclic groups of prime power orders that divide $d$ (Thm.~\ref{Hauptaussage}).
In the case of {\em $\mathcal S$-modules over $K$ in $\RR^d$} (cf.~ Def.~\ref{Def module over K}), where 
$K$ is the quotient field of $\mathcal S$, 
the factor group is either trivial or an elementary Abelian $2$-group, depending on the parity of 
$d$.  This includes the standard icosahedral modules and the rings of cyclotomic integers in 
complex cyclotomic fields.

\section{$\ZZ$-modules}
Let us begin by recalling some well-known facts on Abelian groups.
If $M$ is an Abelian group and $N\subset M$ a subgroup of finite index $[M:N]=k$, then a 
direct consequence of Lagrange's Theorem is that $kM$ is a subgroup of $N$.

\begin{fact}\label{PID-modules}
  Let $M$ be a free $\ZZ$-module of finite rank.
  \begin{enumerate}
    \item \label{f1 PID-modules}
    If $N$ is a submodule of $M$, then $N$ is also a free $\ZZ$-module with 
    $\rk(N)\klg \rk(M)$.
    \item \label{f2 PID-modules}
    If $N$ is a submodule of $M$ of finite index, then $N$ has the same rank as $M$.
  \end{enumerate}
\end{fact}
\begin{proof}
      Cf.~\cite[Thm.~6.2]{aw} for \eqref{f1 PID-modules}.   
      If $[M:N]= k \in \NN$, then $kM\subset N$. If $\{m_1,\ldots,m_r\}$ is a $\ZZ$-basis of
      $M$, 
      then $\{km_1,\ldots,km_r\}$ forms a $\ZZ$-basis for the free module $kM$. 
      Using \eqref{f1 PID-modules} yields 
        $\rk(M)= \rk(kM) \klg \rk(N) \klg \rk(M).$
\end{proof}
\begin{fact}{\cite[Ch.\ 2, Lemma 6.1.1]{BS}}\label{det}
  If $M$ is a torsion-free Abelian group of rank $r$, and $N$ is a subgroup which is also of 
  rank $r$, then the index $[M:N]$ is finite and equals the absolute value of the determinant of 
  the transition matrix from any basis of $M$ to any basis of $N$.\qed
\end{fact}

Together with Fact~\ref{PID-modules}\eqref{f2 PID-modules}, this gives the following equivalence.
\begin{lemma}\label{rank index}
  Let $M\subset \RR^d$ be a free $\ZZ$-module of finite rank and let $N\subset M$ be a submodule.
  Then the index $[M:N]$ is finite if and only if $\rk(N) = \rk(M)$. \qed
\end{lemma}

 Two free $\ZZ$-modules $M,M'\subset\RR^d$ of finite rank are called \emph{commensurate} 
 if their intersection has finite (subgroup) index both in $M$ and in $M'$. In this case we write 
 $M \sim M'$.
 Furthermore, if $M\sim M'$, then $M\cap M'$ is a free $\ZZ$-module with
 $\rk(M)=\rk(M')=\rk(M\cap M')$ by Facts~\ref{PID-modules}\eqref{f1 PID-modules} and 
 \ref{PID-modules}\eqref{f2 PID-modules}.\label{rank commensurate}
  


\begin{lemma}\label{comm}
  Commensurability of free $\ZZ$-modules of finite rank $r$ contained in $\RR^d$ is
  an equivalence relation.
\end{lemma}
\begin{proof}
  Reflexivity and symmetry are clear by definition. For the
  transitivity, let $M_1\sim M_2$ and $M_2\sim M_3$. 
  In particular, the indices $s_{12}=[M_2:(M_1\cap M_2)]$ 
  and $s_{23}=[M_2:(M_2\cap M_3)]$ are
  finite. One obtains $s_{12}M_2\subset (M_1\cap M_2)$ and
  $s_{23}M_2\subset (M_2\cap M_3)$. Since $M_1,M_3$ and
  $s_{12}s_{23}M_2$ are free $\ZZ$-modules of rank $r$ in
  $\RR^d$, Lemma~\ref{rank index} together with the relation
  $$
  s_{12}s_{23}M_2\subset (M_1\cap M_2\cap M_3)
  $$
  now implies that $M_1\cap M_2\cap M_3$ is of finite index
  both in $M_1$ and $M_3$. As a consequence, one obtains 
  $M_1\sim M_3$. 
\end{proof}

Alternatively to the definition of $\os(M)$ above, one easily verifies that
  \begin{equation}\label{zwei Defs os}
    \os(M) = \bmidb{R\in \og(d,\RR)}{\beta RM \subset M \text{ for some } \beta \in \RR_+},
  \end{equation}
where $\RR_+$ denotes the set of positive real numbers.

Therefore, $\os(M)$ consists of all linear isometries that arise from similarity mappings of $M$ into 
itself. We call a submodule of $M$ of the form $\beta RM$ a {\em similar submodule of $M$}. 
(For similar submodules in four dimensions, see \cite{Moo}.)
\begin{lemma}\label{os group}
  Let $M\subset\RR^d$ be a free $\ZZ$-module of finite rank $r$ that spans $\RR^d$. The sets $\os(M)$ and 
  $\oc(M)$ are subgroups of $\og(d,\RR)$.
\end{lemma}
\begin{proof}
  Let $R, S \in \os(M)$. Due to equation~\eqref{zwei Defs os}, there exist positive real numbers 
  $\alpha, \beta$ with 
  $\alpha RM \subset M$ and $\beta S M \subset M$.
  Hence $RS\in \os(M)$, because 
  \begin{equation*}
    \alpha\beta RSM = \alpha R (\beta SM) \subset \alpha R M \subset M.
  \end{equation*}
  One also has $M \subset \alpha^{-1}R^{-1}M$, which implies that the group index
  $[\alpha^{-1}R^{-1}M:M]=s$ is finite by Lemma~\ref{rank index}. Thus $s\alpha^{-1}R^{-1}M \subset M$ is 
  also of finite index. This shows $R^{-1} \in \os(M)$. 
  For the group property of $\oc(M)$ let $R_1, R_2 \in \oc(M)$.  
   $R_2M\sim M$ yields  $M\sim R_2^{-1}M$, and hence
  $R_1M\sim R_1R_2^{-1}M$. On the other hand, $M\sim R_1M\sim R_1R_2^{-1}M$ implies 
  $M \sim R_1R_2^{-1}M$, because commensurability is transitive by Lemma~ \ref{comm}. 
\end{proof}
Let us briefly turn to the subgroup of orientation preserving similarity isometries 
$\sos(M)\subset \os(M)$, which are by definition those similarity isometries $R$ with $\det(R)=1$. 
For planar lattices, these similarity rotations are rather well understood; cf.~\cite{Baa}.
A $\ZZ$-module in an algebraic number field $K$ of degree $n$ is called {\em full}, 
if it contains $n$ linearly independent elements over $\QQ$.
If a full $\ZZ$-module in $K$ is a ring and contains the number $1$, it is called an 
{\em order of $K$}. Any order of $K$ is contained in the ring of algebraic integers $\mathcal O_K$ of $K$,
which is itself an order. Hence $\mathcal O_K$ is also called the \emph{maximal order of $K$}; 
cf.~\cite{BS}.
In the following results on orders of imaginary algebraic number fields, we parametrise the 
Euclidean plane by the complex numbers $\CC$ and furthermore, we use $\so(2,\RR) \simeq \s^1$.

\begin{lemma}\label{sos Ordnung allg}
  Let $K$ be an imaginary algebraic number field and let $\OO$ be an order of $K$. 
  Then 
  \begin{equation*}
    \sos(\OO) 
         =\big\{a/|a| \big | a \in \OO\setminus\{0\}\big\}.
  \end{equation*}
\end{lemma}
\begin{proof}
  For $0\not=a \in \OO$, one has
  $ |a|\cdot a/|a|\OO\subset \OO$,
  because $\OO$ is a ring. Hence $a/|a| \in \sos(\OO)$.
  Conversely, let $r \in \sos(\OO)$, meaning $r \in \s^1$ with 
  $\lambda r \OO \subset \OO$ for 
  some $\lambda \in \RR_+$. Since $1 \in \OO$, this yields
  $\lambda r \in \OO$, say $\lambda r = \beta$.
  Thus $|\lambda|=|\beta|$, because $r \in \s^1$.
  This shows that $r = \pm\beta/|\beta|$ is an $\OO$-direction.
\end{proof}
 There is a close connection between similar submodules of orders $\OO$ of algebraic number fields
 $K$ that arise from rotations and the principal ideals of these orders. 
 The special cases where $K$ is an $n$-th cyclotomic 
 field of class number $1$ and $\OO=\ZZ[\zeta_n]$ (with $\zeta_n$ an $n$-th primitive root of unity),
 or where $\OO$ is a planar lattice with non-generic multiplier ring, can be found in 
 \cite{Gri} and \cite{Baa}.  
\begin{theorem}
  Let $K$ be an imaginary algebraic number field 
  and let $\OO$ be an order of $K$.
  Then the similar submodules of $\OO$ of the form $\alpha R\OO$ 
  with $R \in \sos(\OO)$ are precisely the 
  principal ideals of $\OO$, i.e., the ideals of the form $\kappa \OO$ with $\kappa \in \OO$.
  Moreover, one has
  \begin{equation*}
     [\OO:\kappa \OO] = |N(\kappa)|,
  \end{equation*}
  where  
  $N$ denotes the field norm of $K$.  
\end{theorem}
\begin{proof}
  Let $R \in \sos(\OO)$ and $\alpha \in \RR_+$ with 
  $\alpha R \OO \subset \OO$. Due to Lemma~ \ref{sos Ordnung allg}, there exists a nonzero
  $\delta \in \OO$ such that $ R = \delta/|\delta|$.
  Then $1 \in \OO$ implies  $\alpha\delta/|\delta| \in \OO$. Hence $\alpha R \OO$ is a principal ideal
  of $\OO$.
  Conversely, for any nonzero $\kappa' \in \OO$ one has $\kappa'\OO \subset \OO$, because $\OO$ is
  a ring. Setting $R'= \kappa'/|\kappa'|$, one has $R' \in \sos(\OO)$ by 
  Lemma~\ref{sos Ordnung allg}, and $|\kappa'|R'\OO \subset \OO$.
   
  The second claim follows by a standard argument in Minkowski theory; cf.~\cite[Sec.~ 3, Ch.~ 2]{BS}.
  Considering a Minkowski representation $x(\OO)$ of $\OO$, one finds
  \begin{equation*}
      [\OO : \kappa \OO] = [x(\OO): x(\kappa\OO)] = |N(\kappa)|.
  \end{equation*}
 \end{proof}

\section{$\mathcal S$-Modules}\label{sec moduln}

Let $\mathcal S \subset \RR$ be a ring with unity that is also a finitely generated $\ZZ$-module, hence
a free $\ZZ$-module of finite rank $r$. Furthermore, let $K$ be the field of fractions of $\mathcal S$.

Throughout this section, let $\g\subset \RR^d$ be a free $\mathcal S$-module of rank $d$ that spans 
$\RR^d$, meaning that it is the $\mathcal S$-span of an $\RR$-basis of $\RR^d$.  
Besides the case $\mathcal S=\ZZ$, where $\g$ is a lattice in $\RR^d$, this also covers many
important examples relevant in quasicrystallography; cf.~Example \ref{Beispiel}.

\begin{remark}
  \begin{enumerateA}
    \item
     Every element of $\mathcal S$ is an algebraic integer and  $K$ is a real algebraic number field.
    \item
     $\mathcal S$ is integrally closed if and only if $\mathcal S$ is the ring of integers $\OO_K$ in 
     $K$.
    \item
     $\g$ is a free $\ZZ$-module of rank $rd$.
    \item
     In fact, by standard results from algebra, rings $\mathcal S$ as above are precisely the subrings of 
     rings of integers in real algebraic number fields
     ; cf.~\cite{Huck} for more on this.

  \end{enumerateA}
\end{remark}
\begin{remark}\label{rank}
  Let $\g_1,\g_2$ be free $\mathcal S$-modules of 
  rank $d$ that span $\RR^d$. 
  If $\g_1$ and $\g_2$ are commensurate, then $\g_1\cap \g_2$ is a 
  free $\ZZ$-module of rank $rd$ (cf.~p.~\pageref{rank commensurate}) and it spans $\RR^d$. 
  Namely, if $\g_1 \sim \g_2$, then one has $m=[\g_1: (\g_1\cap\g_2)]<\infty$.
  Hence $m\g_1\subset (\g_1 \cap \g_2)$, which implies that $\g_1 \cap \g_2$ contains an 
  $\RR$-basis of $\RR^d$.
\end{remark}

The following result is of fundamental importance; compare~\cite[Thm. 2.1]{ZOU} 
for the special case $\mathcal S = \ZZ$ ($K=\QQ$).

\begin{theorem}\label{charac}
  Let $\varGamma_1,\varGamma_2\subset\RR^d$ be free $\mathcal S$-modules of rank $d$ that span
  $\RR^d$. Further, let $B_1, B_2\in\operatorname{GL}(d,\RR)$
  be basis matrices of the $\mathcal S$-modules $\varGamma_1$ and $\varGamma_2$,
  respectively. Then, one has
  $$
  \varGamma_1\sim \varGamma_2 \,\,\Leftrightarrow \,\, B_2^{-1}B_1\in\operatorname{GL}(d,K)\,.
  $$   
\end{theorem}
\begin{proof}Let firstly $\g_1 \sim \g_2$.
  By Remark~ \ref{rank}, the intersection $\g'=\g_1\cap \g_2$ contains an $\RR$-basis $\mathcal B$ 
  of $\RR^d$. 
  Let $B\in\operatorname{GL}(d,\RR)$ be the associated matrix. Then
  there exist non-singular matrices
  $Z_1,Z_2\in\operatorname{Mat}(d,\mathcal S)$ such that
  $$
  B_1Z_1=B=B_2Z_2\,,
  $$
  whence $B_2^{-1}B_1=Z_2Z_1^{-1}\in\operatorname{GL}(d,K)$ by the
  standard formula for the inverse of a matrix. Conversely, if
  $B_2^{-1}B_1\in\operatorname{GL}(d,K)$, then there is a non-zero number
  $s\in \mathcal S$ such that $B=s B_2^{-1}B_1\in\operatorname{Mat}(d,\mathcal S)$. 
  Setting $\varGamma'=\varGamma_1\cap
  \varGamma_2$, the identity $s B_1=B_2B$ implies that 
  $s\varGamma_1\subset \varGamma'\subset\varGamma_1$. 
  Since $s\varGamma_1$ and
  $\varGamma_1$ are both free $\ZZ$-modules of rank $rd$, one obtains
  $[\varGamma_1:\varGamma']<\infty$. By symmetry, one also
  has $[\varGamma_2:\varGamma']<\infty$. Hence, $\varGamma_1\sim \varGamma_2$.
\end{proof}
\begin{definition}
 For an arbitrary element $R \in \og(d,\RR)$, define
  \begin{equation*}
     \sca(R)= \bmidb{\alpha \in \RR}{\varGamma \sim \alpha R \varGamma}.
  \end{equation*}
\end{definition}
Note that $\os(\g) = \bmidb{R \in \og(d,\RR)}{\sca(R)\not = \varnothing}$.

\begin{remark}\label{scal mult}
  If $\beta \in \sca(R)$, then there exists 
  a nonzero element $t \in \ZZ$ such that $t\beta R \g \subset \g$. 
  For if $\beta \in \sca(R)$, then the index $[\beta R\g :(\g\cap \beta R\g)]=t$ is finite and one has 
      $t\beta R\g \subset (\g\cap \beta R\g) \subset \g$.
 \end{remark}
 
\begin{lemma} \label{delta moduln} 
  For all elements $\alpha \in \sca(R)$ one has $\alpha^d \in K$. Thus $\alpha$ is an 
  algebraic number. 
\end{lemma}
\begin{proof}
  One has $\alpha R \g \sim \g$ by assumption. Let $B$ be a basis matrix for $\g$. Then 
  $\alpha R B$ is a basis matrix for $\alpha R \g$. By Theorem~ \ref{charac}, one has 
  $ B^{-1}\alpha R B  \in \operatorname{GL}(d,K)$, 
  which immediately yields 
  $$\alpha^d=\pm\det(\alpha R) \in K.$$ 
  Hence $\alpha$ is algebraic over $K$, which implies that $K(\alpha)$ is a finite field 
  extension of $K$, and thus also of $\QQ$.
  Therefore $\alpha$ is algebraic over $\QQ$.
\end{proof}

Let $\RR^\bullet$ denote the multiplicative group formed by the nonzero real numbers. Denoting by  
$\RR^\bullet \gl(d,K)$ the group consisting of all elements of the form $t H$ with 
$t \in \RR^\bullet$ and $H \in \gl(d,K)$, there is the following consequence of Theorem~\ref{charac}.
\begin{cor}\label{cor charac}
   For any basis matrix $B_\g$ of $\g$, one has
  \begin{equation}\label{eq countable}
    \osg =\big( B_\g^{\vphantom{-1}} \li( \RR^\bullet \gl(d,K)\re)B_\g^{-1} \big)\cap \og(d,\RR)
  \end{equation}
  and
  \begin{equation*}
    \ocg =\big( B_\g ^{\vphantom{-1}}\gl(d,K) B_\g^{-1}\big)\cap \og(d,\RR).
  \end{equation*} 
\end{cor}
\begin{proof}
  Let $\{\gamma_1,\ldots,\gamma_d\}$ be an $\mathcal{S}$-basis of $\g$ and denote by $B_\g$ the 
  associated matrix.
  For $R\in \osg$ there exists a positive real number $\alpha$ with $\alpha R \g \sim \g$.
  The set $\{\alpha R\gamma_1,\ldots,\alpha R\gamma_d\}$ is an 
  $\mathcal{S}$-basis of $\alpha R \g$ with associated matrix $B_{\alpha R \g}=\alpha R B_\g$.
  Theorem~\ref{charac} then implies that there exists an $ H \in  \gl(d,K)$ with
  $ H=B_\g^{-1}\alpha R B_\g $.
  Thus, one has $R= B_\g \alpha^{-1}H B_\g^{-1}$.

  If on the other hand $S\in \og(d,\RR)$ and 
     $ S= B_\g \beta J B_\g^{-1}$
  for some $\beta \in \RR^\bullet$ and $J \in \gl(d,K)$, then one has
  \begin{equation*}
    B_\g^{-1}B_{\beta^{-1}S\g} = B_\g^{-1} \beta^{-1}S B_\g\,\, \in \gl(d,K).
  \end{equation*}
  Theorem~\ref{charac} therefore implies $\beta^{-1}S\g \sim \g$, which shows $S \in \osg$.
\end{proof}

\begin{remark}\label{rem countable}
  By Corollary~\ref{cor charac}, every element $R \in \osg$ can be written as $R= B_\g\beta H B_\g^{-1}$ 
  with $\beta \in \RR^\bullet$ and $H \in \gl(d,K)$.
  Theorem ~\ref{charac} implies $\beta R \g \sim \g$ and hence $\beta \in \sca(R)$, which shows that 
  $\beta$ is an algebraic number and that $\beta^d \in K$ by Lemma~\ref{delta moduln}. 
  But the set of all algebraic numbers is 
  countable and so is $\gl(d,K)$. Therefore the group $\osg$ is countable and in particular, the 
  subgroup $\ocg$ is countable as well. 
  The explanations above imply that, in Corollaries~\ref{cor charac} and \ref{cor2 charac}, 
  $\RR^\bullet$ can be replaced by the set of all nonzero real numbers $\delta$ with $\delta^d \in K$.
\end{remark}

\begin{cor}\label{cor2 charac}
  Let $\g \subset K^d$. One has 
  \begin{equation*}
    \osg = \big(\RR^\bullet \gl(d,K)\big)\cap \og(d,\RR)
  \end{equation*}
   and
  \begin{equation*}
     \ocg = \og(d,K).
  \end{equation*}
\end{cor}
\begin{proof}
  The assumption $\g \subset K^d$ yields $B_\g \in \gl(d,K)$. Therefore, one has 
  $ B_\g ^{\vphantom{-1}}\gl(d,K) B_\g^{-1}= \gl(d,K)$, and the claim follows from 
  Corollary ~\ref{cor charac}.
\end{proof}

\begin{lemma}\label{scal}
 For $R \in \os(\varGamma)$, the following 
 assertions hold. 
  \begin{enumerate}
    \item\label{rel1 scal}    
       $ b\cdot \sca(R) = \sca(R)$ \/ for all \/ $b \in K\setminus\{0\}$     
    \item\label{rel3 scal}
       $ r\varGamma \sim \varGamma$ \/ with \/ $r\in \RR$ \/ implies \/ $r \in K$
    \item \label{rel2 scal}
       $ \alpha\beta^{-1} \in K$ \/ for all \/ $\alpha, \beta \in \sca(R)$       
  \end{enumerate}
\end{lemma}
\begin{proof}
 Let $\alpha \in \sca(R)$.
 Since $K$ is the field of fractions of $\mathcal S$, every nonzero element $b\in K$ can be 
 written as
 $b=b_1/b_2$ with $b_1,b_2 \in \mathcal S\setminus\{0\}$. Then, one finds
 \begin{equation*}
    \frac{b_1}{b_2}\alpha R \varGamma \sim \frac{1}{b_2} \alpha R \varGamma 
         \sim \frac{1}{b_2}\g.
 \end{equation*}
  One easily observes that $\frac{1}{b_2}\g \sim \g$.
  Hence $b\alpha R \varGamma \sim \varGamma$, yielding $ b\cdot \sca(R) \subset \sca(R)$. Thus, one 
  also has $ b^{-1}\cdot \sca(R) \subset \sca(R)$, which proves \eqref{rel1 scal}. 
  In order to show \eqref{rel3 scal}, let $u\in \RR$ with $u\g\sim\g$. Due to 
  Remark~ \ref{scal mult}, there exists a nonzero integer $k$ such that $ku\g \subset \g$.
  Let $\gamma \in \g$ be represented in terms of an $\mathcal S$-basis 
  $\gamma_1,\ldots,\gamma_d$ of $\g$ as 
  $\gamma = \sum_{i=1}^d c_i\gamma_i$ with $c_i \in \mathcal S$.
  On the other hand, $ku\gamma$ can be represented as 
  $ku\gamma = \sum_{i=1}^d a_i\gamma_i$, where $a_i \in \mathcal S$. Thus
  $$\sum_{i=1}^d kuc_i\gamma_i = \sum_{i=1}^d a_i\gamma_i.$$
  By assumption, $\g$ spans the $\RR^d$. Hence $\{\gamma_1,\ldots,\gamma_d\}$ forms an 
  $\RR$-basis of $\RR^d$.
  Therefore, one has $kuc_i=a_i$, yielding $u=a^{\vphantom{-1}}_ic_i^{-1}k^{-1} \in K$. 
  Finally, \eqref{rel2 scal} is obtained from \eqref{rel3 scal} as follows. By assumption, 
  one has
  \begin{equation*}
     \beta R\varGamma \sim \varGamma \sim \alpha R \varGamma.
  \end{equation*}
  Multiplying with $1/\beta$ gives       
    $ R \varGamma \sim \frac{\alpha}{\beta} R \varGamma$,
  which completes the proof.  
 \end{proof}
In accordance with the previous notation, denote by $K^\bullet$ the multiplicative group formed by the 
nonzero elements of $K$.
\begin{remark}\label{rem scal}
  As a direct consequence of Lemma~\ref{scal}, one has
       $\sca(R) = \alpha K^\bullet$
  for any $\alpha \in \sca(R)$.
\end{remark}
 
 Define \label{eta} the map
\begin{equation*} 
   \eta\colon \os(\varGamma) \longrightarrow \RR^\bullet/K^\bullet\\
\end{equation*}
by
\begin{equation*}
   \qquad \,\, R \longmapsto \sca(R).
\end{equation*}
This map is well-defined due to the fact that $\sca(R)$ is non-empty for $R \in \os(\varGamma)$
and by Remark~\ref{rem scal}. 
\begin{lemma}\label{eta hom}
  The map $\eta$ is a group homomorphism with $\Ker(\eta)=\oc(\varGamma)$.
\end{lemma}
\begin{proof}
 Let $R,S \in \os(\varGamma)$ and $\alpha \in \sca(R)$,
 $\beta \in \sca(S)$. We need to show that $\alpha\beta \in \sca(RS)$.
 By assumption, one has
 \begin{equation*}
   \varGamma \sim \alpha R \varGamma \sim \alpha R (\beta S \varGamma) = \alpha\beta RS\varGamma.
 \end{equation*}
  Thus $\alpha\beta \in \sca(RS)$, hence $\eta$ is a group homomorphism.
It remains to show that $\Ker(\eta)=\oc(\varGamma)$.
For $R \in \oc(\varGamma)$ the set $\sca(R)$ contains $1$, thus $ \eta (R) = \sca(R) = K^\bullet$ by
Remark~\ref{rem scal}.
Conversely, if $S \in \Ker(\eta)$, one has $\sca(S)= K^\bullet$, which implies $S \in \oc(\varGamma)$.
\end{proof}

As the kernel of a group homomorphism, $\oc(\varGamma)$ is a normal subgroup of 
$\os(\varGamma)$.
The factor group $\osg/\ocg$ is isomorphic to the image of
$\eta$, which is a subgroup of $\RR^\bullet/K^\bullet$ and thus Abelian.
Furthermore, $\osg/\ocg$ is countable by Remark~\ref{rem countable}. 
The corresponding result holds for the special case of orientation-preserving isometries by considering 
the restriction of $\eta$ to $\sos$.

To unfold the structure of the factor group $\osg/\ocg$, we need the following result from the theory 
of Abelian groups.
\begin{prop}{\cite[Thms.\ 5.1.9 and 5.1.12]{Scott}}\label{direct sum} Let G be an Abelian group.
 \begin{enumerate}
 \item\label{direct sum 1}
   If a prime number $p$ exists such that $x^p=1$ for all $x \in G$, then $G$ is the 
   direct sum of subgroups of order $p$.  
 \item\label{direct sum 2}
   If a positive integer $n$ exists such that $x^n=1$ for all $x \in G$, then 
   G is the direct sum of cyclic groups of prime power orders. \qed
 \end{enumerate}
\end{prop}

\begin{theorem}\label{Hauptaussage} 
  The group $\osg/\ocg$ is the direct sum of cyclic groups of 
  prime power orders that divide $d$.
\end{theorem}
\begin{proof}
  We consider again the group homomorphism $\eta\colon \osg \longrightarrow 
  \RR^\bullet/K^\bullet$.  Let $R \in \osg$.
  According to Lemma~ \ref{delta moduln}, one has $\alpha^d \in K^\bullet$ for any nonzero 
  $\alpha \in \sca(R)$, 
  which yields  
  \begin{equation}\label{delta^d is 1}
    \eta(R)^d=\sca(R)^d=(\alpha K^\bullet)^d = \alpha^d K^\bullet = K^\bullet
  \end{equation}
  in $\RR^\bullet/K^\bullet$.
  Using the group isomorphism $\eta(\osg) \simeq \osg/\ocg$, this shows that the order of each 
  element of $\osg/\ocg$ divides $d$.
  Proposition ~\ref{direct sum}\eqref{direct sum 2} then implies that 
  the group 
  $\osg/\ocg$ is the direct sum of cyclic groups of prime power orders. 
  Consequently, the prime power order of each cyclic group divides $d$.
\end{proof}
\begin{example}
  Denote by $\{e_1,\ldots, e_d\}$ the canonical
  basis of $\RR^d$.  Let $n\grg 1$ be a natural number with
  $\xi=\sqrt[d]{n} \not \in \QQ $. 
  The $\ZZ$-span $\g$ of $\bmidb{\xi^ie_i}{1\klg i\klg d}$
  is a lattice in $\RR^d$.  Consider the cyclic permutation
  $\sigma=(12....d)$ of the symmetric group $ S_d$. Then $\sigma$
  induces a linear isomorphism $R$ of $\RR^d$ by permuting the
  canonical basis vectors, i.e., $Re_i=e_{\sigma(i)}$.  Since
  $\xi R\g \subset \g$, $R$ is a similarity isometry of $\g$.  But
  $R$ is not a coincidence isometry of $\g$ (because $\sqrt[d]{n} \not
  \in \QQ$). Setting $m=\min\limits_{i}{\bmidb{\xi^i}{\xi^i \in \QQ}}$,
  one easily verifies that the factor group $\osg/\ocg$ contains the
  cyclic group $C_m$ of order $m$ generated by the equivalence class of $R$.
  If $m$ is not a prime power, then the fundamental theorem of finitely generated
  Abelian groups states that $C_m$ is the direct sum of cyclic groups of prime power orders.
  Examples of the module case can be constructed similarly.
\end{example}
\begin{cor}
  If $d=p$ is a prime number, then $\osg/\ocg$ is an elementary
  Abelian p-group, i.e., it is the direct sum of cyclic groups of order $p$. \qed
\end{cor}
\begin{definition} \label{Def module over K}
  Denote by $\langle .\,,.\rangle$ the standard scalar product of $\RR^d$. 
  We call $\g$ an  {\em $\mathcal S$-module over $K$ in $\RR^d$} if it satisfies
  $\langle \gamma,\gamma\rangle \in K$ for all $\gamma \in \g$.
\end{definition}
\begin{example}\label{Beispiel}
 \begin{enumerate}
   \item
   Let $\mathcal S = \ZZ$. Then, $K = \QQ$ and the $\mathcal S$-modules $\varGamma$ over $K$ in $\RR^d$ 
   are precisely the \emph{rational lattices} in $\RR^d$; cf.~\cite{crs} for examples.
   \item
   For $n\in\NN$, $\mathcal{S}^n$ is an $\mathcal S$-module over $K$ in $\RR^n$.
   \item 
   Consider the quadratic number field $L:=\LQ(\tau)$, where $\tau$ is the
   golden ratio, i.e., $\tau=(1+\sqrt{5})/2$. Then, one has $\mathcal{O}_L=\LZ[\tau]$. 
   The icosian ring 
   $$\mathbbm{I}=\langle (1,0,0,0),(0,1,0,0),\tfrac{1}{2}(1,1,1,1),\tfrac{1}{2}(1-\tau,\tau,0,1)
        \rangle_{\mathcal{O}_L}\subset L^4$$ 
   is an $\mathcal{O}_L$-module over $L$ in $\RR^4$ (see~ \cite{Moo}).
   Further, both the standard body centred
   icosahedral module $\mathcal{M}_{\rm B}$ and the standard face centred
   icosahedral module $\mathcal{M}_{\rm F}$ of quasicrystallography are $\mathcal{O}_L$-modules
   over $L$ in $\RR^3$; cf.~\cite{baake,bpr} and references therein. 
   \item
   Consider the quadratic number field $L:=\LQ(\sqrt{2})$. Then,
   $\mathcal{O}_L=\LZ[\sqrt{2}]$ and further, the octahedral (or cubian) ring
   $$\mathbbm{K}=
       \langle (1,0,0,0),\tfrac{1}{\sqrt{2}}(1,1,0,0),\tfrac{1}{\sqrt{2}}(1,0,1,0),\tfrac{1}{2}(1,1,1,1)
   \rangle_{\mathcal{O}_L}\subset L^4$$ 
   is an $\mathcal{O}_L$-module
   over $L$ in $\RR^4$; cf.~\cite{Moo,bpr} and references therein. 
   \item\label{cyclotomic}
   Consider the complex cyclotomic field $\LQ(\zeta_m)$, where $m\geq 3$ and $\zeta_m$ is a 
   primitive
   $m$th root of unity in $\CC$ (e.g. $\zeta_m=e^{2\pi i/m}$). Recall that $\LQ(\zeta_m)$ is a finite 
   Galois extension of $\LQ$ with maximal real subfield 
   $L:=\LQ(\zeta_m)\cap\RR=\LQ(\zeta_m+\bar{\zeta}_m)$; cf.~\cite[Theorem 2.5]{Wa}. Further, it is well 
   known that $\mathcal{O}_{\LQ(\zeta_m)}=\LZ[\zeta_m]$ and
   $\mathcal{O}_{L}=\LZ[\zeta_m+\bar{\zeta}_m]$; cf.~\cite[Theorem 2.6 and Proposition 2.16]{Wa}. 
   Moreover, since $\zeta_m^2=(\zeta_m+\bar{\zeta}_m)\zeta_m-1$,
   the ring $\LZ[\zeta_m]$ is the $\LZ[\zeta_m+\bar{\zeta}_m]$-span of the $\RR$-basis 
   $\{1,\zeta_m\}$ of $\CC$. Identifying the complex numbers $\CC$ with $\RR^2$, one can now verify that 
   $\LZ[\zeta_m]$ is an
   $\mathcal{O}_L$-module over $L$ in $\RR^2$. 
   In particular, rings of integers in complex cyclotomic fields can be used to construct planar 
   mathematical quasicrystals such as the vertex sets of Penrose, Ammann-Beenker or shield tilings; 
   cf.~\cite{Mood, Grimm, Grim, Plea}. 
\end{enumerate}
\end{example}
  
\begin{theorem}\label{Skalarprodukt}
   For an $\mathcal S$-module $\g$ over $K$ in $\RR^d$ one has
  \begin{equation*}
    \osg^2 \subset \ocg,
  \end{equation*}
 where $\osg^2=\bmidb{R^2}{R\in \osg}$.
 Thus, the factor group $\osg/\ocg$ is an elementary Abelian $2$-group, when $d$ is even.
 If $d$ is odd, one has $\osg = \ocg$.
\end{theorem}
\begin{proof}
  Let $R \in \osg$. Then there exists an element $\alpha \in \RR_+$ with
   $\alpha R \g \subset \g$. 
  By assumption, one has
   $\langle \alpha R \gamma,\alpha R \gamma\rangle \in K$ for all $\gamma \in \g$. 
  Hence
  $\alpha^2\in K^\bullet$, say $\alpha^2=s_1/s_2$, where 
  $s_1, s_2\in\mathcal S\setminus\{0\}$. 
  Since $s_2\alpha^2=s_1\in\mathcal S$ and $\alpha R\g\subset\g$, 
  this yields
  $$
  \g\supset s_2\alpha R(\alpha R\g)=s_2\alpha^2R^2\g
  \subset R^2\g,
  $$
  whence
  $$
  s_1R^2\g\subset (\g \cap R^2\g).
  $$
  Since $\g$, $R^2\g$ and $s_1R^2\g$ are $\ZZ$-modules of the same finite rank, one obtains
  that both $[\g:s_1R^2\g]$ and $[R^2\g:s_1R^2\g]$ are finite. It follows that
  $\g\sim R^2\g$, meaning that $R^2$ is a coincidence
  isometry of $\g$. Consequently, $\osg^2\subset\ocg$.
  Thus, every element of the factor group $\osg/\ocg$ is of order $1$ or $2$. 
  By Proposition~ \ref{direct sum}\eqref{direct sum 1}, the factor group is an elementary Abelian 
  $2$-group.
 
  If $d$ is odd, set $d = 2m+1$ with $m \in \NN$. Then
   \begin{equation*}
     \alpha(\alpha^2)^m = \alpha^d \,\, \in K^\bullet
   \end{equation*}
   yields $\alpha \in K^\bullet$, because $\alpha^2 \in K^\bullet$. 
   Thus $\eta(R)=\sca(R)=\alpha K^\bullet = K^\bullet$ in $\RR^\bullet/K^\bullet$ for all $R \in \osg$, 
   whence 
   $\osg/\ocg$ is the trivial group.
   In other words, one has $\osg = \ocg$ for $d$ odd.
\end{proof}
\begin{example}
  Using the notation of Example~\ref{Beispiel}\eqref{cyclotomic},
  consider a cyclotomic field $\QQ(\zeta_m)$ of class number one, meaning that
  its ring of integers $\ZZ[\zeta_m]$ is a unique factorization domain; see \cite{Roth} for more on this.
  $\ZZ[\zeta_m]$ is an $\mathcal{O}_L$-module
  over $L$ in $\RR^2$.  Since we are working in $2$-space,
  Theorem~\ref{Skalarprodukt} implies that the factor group of
  similarity modulo coincidence isometries is an elementary Abelian
  $2$-group. Combining the results of \cite{Roth} and \cite[Prop.~1.89]{Huc}, one immediately obtains
  \begin{equation*}
     \sos(\ZZ[\zeta_m])/\soc(\ZZ[\zeta_m])\simeq C_2\times C_2^{(\aleph_0)},
  \end{equation*}
  where $\sos$ and $\soc$ indicate the restriction to orientation-preserving isometries, and 
  $ C_2^{(\aleph_0)}$ stands for the direct sum of countably many cyclic groups of order $2$.
\end{example}

\section{Acknowledgments}
  The author is thankful to M.\ Baake, C.\ Huck, R.\ V.\ Moody and P.\ Zeiner for valuable discussions 
  and comments on the manuscript. 
  This work was supported by the German Research Council (DFG), within the CRC 701.

\end{document}